\documentclass[en, secnum]{./cls/myamspaper}
\usepackage[backend=biber]{biblatex}
\renewcommand{\orb}{\operatorname{\overline{orb}}}
\calclayout

\title[Decompositions of dynamical systems via the Koopman operator]{Decompositions of dynamical systems induced by the Koopman operator}

\date{\today}

\usepackage{tikz}
\usepackage{stmaryrd}

\begin{document}

\begin{abstract}
  \setlength\parindent{0em}
  \setlength\parskip{0.3\baselineskip}

For a topological dynamical system we characterize the decomposition of the state space induced by the fixed space of the corresponding Koopman operator. For this purpose, we introduce a hierarchy of generalized orbits and obtain the finest decomposition of the state space into absolutely Lyapunov stable sets. Analogously to the measure-preserving case, this yields that the system is topologically ergodic if and only if the fixed space of its Koopman operator is one-dimensional.
\end{abstract}


\maketitle

\section{Introduction}

It is a common strategy to decompose a dynamical system into smaller parts and investigate these instead of the whole system. There exists a variety of such decompositions, e.g., by Conley (see \cite{conley} or \cite{norton}), the decomposition of the chain-recurrent set into chain components (see, e.g., \cite{shimomura}) or orbit-closure decompositions in \cite{gottschalk} to name a few.

In this paper we study a new decomposition of a \emph{topological dynamical system} $(K;\varphi)$, consisting of a compact Hausdorff space $K$ and a continuous map \[\varphi\colon K\to K.\] To do so, we consider the corresponding \emph{Koopman operator} \[\T f :=f\circ\varphi\] on the $C^*$-algebra $\mathrm{C}(K)$ of all continuous complex-valued functions on $K$ and its fixed space \[\fix\T:=\left\{f\in\mathrm{C}(K)\colon\T f=f\right\}.\] This fixed space yields a decomposition of $K$ into disjoint $\varphi$-invariant and closed sets (see Section \ref{sec2}). To characterize this dynamically, we introduce a transfinite hierarchy of \emph{generalized $\varphi$-orbits}. Moreover, we show that $\fix\T$ induces the finest decomposition of $K$ into \emph{absolutely Lyapunov stable} subsets (see Theorem \ref{absliap}). 

As a consequence, we obtain that this decomposition is trivial, i.e., the fixed space of $\T$ has dimension $1$, if and only if the system $(K;\varphi)$ is \emph{topologically ergodic} meaning that there exists $x\in K$ with generalized orbit $\sorb(x)=K$. This is by analogy with a measure-preserving dynamical system $(\Omega, \Sigma, \mu;\varphi)$ being \emph{ergodic} (i.e., indecomposable) if and only if the fixed space \[\fix\T:=\left\{f\in \mathrm{L}^1(\Omega, \Sigma, \mu)\colon\T f=f\right\}\] of the corresponding Koopman operator on $\mathrm{L}^1(\Omega, \Sigma, \mu)$ is one-dimensional. A variety of examples demonstrate the complexity of the topological situation. I thank Nikolai Edeko for providing some of them and Roland Derndinger for many helpful discussions.

\section{The decomposition of $K$ corresponding to $\fix\T$}

\label{sec2}
Since the fixed space $\fix\T$ is a $\T$-invariant $C^*$-subalgebra of $\mathrm{C}(K)$, the Gelfand-Naimark theorem shows that it is isomorphic to a space $\mathrm{C}(L)$ for some compact Hausdorff space $L$, called the \emph{fixed factor} or \emph{maximal trivial factor} of $K$ (see \cite{nied}). The embedding $\mathrm{C}(L)\hookrightarrow \mathrm{C}(K)$ yields a surjection $p\colon K\twoheadrightarrow L$, called \emph{factor map} (see, e.g., \cite[Chapter 2.2]{efhn}). This induces a disjoint splitting

\[
K=\dot{\bigcup\limits_{l\in L}}\,p^{-1}(\{l\})
\]
into closed $\varphi$-invariant sets, hence an equivalence relation $\sim$ on $K$ with equivalence classes $p^{-1}(\left\{l\right\})$, $l\in L$. Our problem is the following.
\begin{problem}
\label{prob}
Describe this equivalence relation by dynamical and topological properties of $(K;\varphi)$.
\end{problem}

For this purpose we introduce some technical terms.

\begin{definition}
\label{levelsets}
\begin{enumerate}
\item
A nonempty set $M\subseteq K$ is called a \emph{level set of $\fix\T$} if $f\vert_M$ is constant for all $f\in\fix\T$.
\item
A level set $M$ is called \emph{maximal} if for any other level set $M'\subseteq K$ with  $M\subseteq M'$  already $M'=M$.
\end{enumerate}
\end{definition}
\begin{remark}
\label{maxlevset}
\begin{enumerate}
\item 
Maximal level sets exist by Zorn's lemma and are closed.
\item 
\label{maxlevsetchar}
A set $M\subseteq K$ is a maximal level set of $\fix\T$ if and only if $M=p^{-1}(\{l\})$ for some $l\in L$. 
\end{enumerate}
\end{remark}

\begin{samepage}
\begin{proposition}
\label{technik}
Let $(K;\varphi)$ be a topological dynamical system and identify $\fix\T$ with $\mathrm{C}(L)$. Let $\sim$ be any equivalence relation on $K$ with canonical projection $\pi\colon K\to K/{\sim}$ satisfying 
\begin{enumerate}[(i)]
\item
\label{inv}
$\varphi(x)\sim x$ and 
\item
\label{lev}
$\pi^{-1}([x])$ is a level set of $\fix\T$
\end{enumerate}
 for all $x\in K$. Then the following are equivalent. 
\begin{enumerate}
\item \label{44} For each $x\in K$ the preimage $\pi^{-1}([x])$ is a maximal level set of $\fix\T$.
\item \label{33} With respect to the quotient topology $K/{\sim}$ is Hausdorff.
\item \label{cc} $K/{\sim}\cong L$.
\end{enumerate}
\end{proposition}
\end{samepage}

\begin{proof}
\ref{44} $\Rightarrow$ \ref{cc}:
By assumption we have $\pi^{-1}([x])=p^{-1}(\{l\})$ for some $l\in L$, where $p\colon K\to L$ is the factor map. Hence $\pi([x])=\pi([y])$ if and only if $p(x)=p(y)$ for $x,y\in K$. By the universal property of the quotient topology there are unique continuous maps $h\colon K/{\sim}\to L$ and $g\colon L\to K/{\sim}$ such that $h\circ\pi=p$ and $g\circ p=\pi$. Then $g=h^{-1}$ since \[g\circ h(\pi(x))=g(p(x))=\pi(x)\] and \[h\circ g(p(x))=h(\pi(x))=p(x)\] for all $x\in K$. Hence $h$ is a homeomorphism between $K/{\sim}$ and $L$. 

\ref{cc} $\Rightarrow$ \ref{33}:
Since $K/{\sim}$ is homeomorphic to the space $L$ it is Hausdorff. 

\ref{33} $\Rightarrow$ \ref{44}: 
It suffices to show that $p^{-1}(p(x))\subseteq \pi^{-1}([x])$ for all $x\in K$, because $p^{-1}(p(x))$ are the maximal level sets of $\fix\T$. Assume there are some $y,z\in K$ such that $y\in p^{-1}(p(z))\backslash \pi^{-1}([z])$. Since $K/{\sim}$ is Hausdorff, $[y]$ and $[z]$ are closed. By Urysohn's lemma, there is some $\tilde{f}\in\mathrm{C}(K/{\sim})$ such that $\tilde{f}([y])\neq\tilde{f}([z])$. Then $f:=\tilde{f}\circ\pi\in\fix\T$, since $f$ is continuous and $x\sim\varphi(x)$ implies \[f(x)=\tilde{f}([x])=\tilde{f}([\varphi(x)])=\T f(x)\] for all $x\in K$. By the universal property of the quotient topology there is some $\hat{f}\in \mathrm{C}(L)$ such that $f=\hat{f}\circ p$. Since $y\in p^{-1}(p(z))$, we obtain \[f(y)=\hat{f}(p(y))=\hat{f}(p(z))=f(z)\]
which contradicts $\tilde{f}([y])\neq\tilde{f}([z])$.
\end{proof}

\begin{remark}
For a topological dynamical system $(K;\varphi)$, the trivial equivalence relations 
\begin{enumerate}
\item $x\sim y$ only for $x=y$ and 
\item
$x\sim y$ for all $x,y\in K$
\end{enumerate}
show that neither $\varphi(x)\sim x$ for $x\in K$ implies that $\pi^{-1}([x])$ is a level set nor the converse implication.
\end{remark}

\section{Equivalence relations induced by generalized orbits}

\label{chap3}
Our goal is to dynamically describe $\fix\T$ for a topological dynamical system $(K;\varphi)$ (cf. Problem \ref{prob}). To do so, we use Lemma \ref{technik} and search for an equivalence relation $\sim$ on $K$ such that $K/{\sim}$ is Hausdorff, $\varphi(x)\sim x$ and $\pi^{-1}([x])$ are level sets for all $x\in K$. 

A first observation is the following. If we take the closed \emph{orbit} \[\orb(x):=\overline{\left\{\varphi^n(x):\,n\in\mathbb{N}_0\right\}}\] for $x\in K$, then $f|_{\orb(x)}$ is constant for all $f\in\fix\T$. Thus, every closed orbit is a level set of $\fix\T$. If $K$ admits a decomposition into mutually disjoint closed orbits, this clearly induces an equivalence relation $\sim$ with $\varphi(x)\sim x$ for all $x\in K$. But the corresponding quotient space may not be Hausdorff as the following example shows.

\begin{example}
\label{kreisvarrot}
Let $K:=D$ be the closed unit disk in $\mathbb{C}$ and \[\varphi(x):=re^{2\pi i(\alpha+r)}\] for $x:=re^{2\pi i\alpha}\in K$ with $r\in [0,1]$, $\alpha\in[0,1)$. Denote by $\mathbb{T}$ the unit circle in $\mathbb{C}$. Then the closed orbits
 \[\orb(x)=\begin{cases}
\left\{re^{2\pi i(\alpha+nr)}:\, n=1,...,q-1\right\}&\text{for $r$ rational, }r=\frac{p}{q}\text{ with }p\text{ and }q\text{ coprime},\\
r\mathbb{T} &\text{for $r$ irrational}
\end{cases}
\]
form a non-trivial decomposition of $K$. However, the fixed space of $\T$ is \[\fix\T=\left\{f\in \mathrm{C}(K):\,f\vert_{c\mathbb{T}}\equiv const.\text{ for all $c\in[0,1]$}\right\},\] so the maximal level sets are the circles $c\mathbb{T}$ for $c\in [0,1]$. This shows that even mutually disjoint closed orbits may induce a quotient space that is not Hausdorff.  
\end{example}

Our strategy to obtain the Hausdorff quotient space corresponding to $\fix\T$ is based on the following characterization.

\begin{remark}
\label{hausd}
A topological space $X$ is Hausdorff if and only if each point is the intersection of its closed neighborhoods, i. e. for all $x\in X$ we have \[\{x\}=\bigcap_{U\in\mathcal{U}(x) \text{ closed}}U\] where $\mathcal{U}(x)$ denotes the neighborhood filter of $x$.
\end{remark}
Moreover, we need the following definition.

\begin{definition}
Let $(K_x)_{x\in K}$ be a covering of $K$ satisfying $x\in K_x$ for all $x\in K$. Define an equivalence relation $\sim$ on $K$ via $x\sim y$ for $x,y\in K$ if there is some $k\in\mathbb{N}$, $x_1,...,x_k\in K$ such that $x_1=x$ and $x_k=y$ and \[K_{x_i}\cap K_{x_{i+1}}\neq\emptyset \text{ for $i=1,...,k-1$}.\] We call $\sim$ the \emph{equivalence relation generated by $(K_x)_{x\in K}$}.
\end{definition}

\begin{remark}
\label{remeqcl}
For the equivalence relation $\sim$ generated by $(K_x)_{x\in K}$ and its canonical projection $\pi: K\to K/{\sim}$, we have $\pi^{-1}([x])=\bigcup_{y\in K, y\sim x}K_y$.
\end{remark}

We now outline our strategy. Starting from a quotient space $K/{\sim_0}$ we successively construct the Hausdorff property by the following steps. We first build the intersection of closed neighborhoods of each equivalence class (cf. Remark \ref{hausd}). The preimages under the canonical projection of these intersections yield a covering of $K$. We obtain a new quotient space $K/{\sim_1}$ taking the equivalence relation generated by this covering. We then repeat the steps above with the new eqivalence relation and so forth. We show that by repeating sufficiently often we arrive at a Hausdorff space.

\begin{remark}
For a similar approach to obtain a \emph{Hausdorffization}, we refer to \cite{Mun}, \cite{kelly} or \cite{osborne}.
\end{remark}

\subsection{Approximating orbits and superorbits}

We apply this strategy to our situation in order to reach the assumptions \ref{inv} and \ref{lev} of Proposition \ref{technik} to characterize the fixed factor of $\T$.

\begin{definition}
\label{aorbsorb}
\begin{enumerate}
\item
\label{aorb1}
We define the \emph{approximating orbit} of $x$ for each $x\in K$ as
\[\aorb (x):=\bigcap\limits_{\substack{U\in \mathcal{U}(x) \text{ closed},\\\varphi(U)\subseteq U }}U\,.\]
\item
\label{sorb1}
Let $\sim$ be the equivalence relation on $K$ generated by $(\aorb(x))_{x\in K}$. The \emph{superorbit} of $x$ is
\[\sorb (x):=\pi^{-1}([x])=\bigcup\limits_{\substack{y\in K,\\ y\sim x}}\aorb(y)\,.\] 
\end{enumerate}
\end{definition}

\begin{proposition}
\label{lemlev}
For each $x\in K$ we have that
\begin{enumerate}
\item\label{inva}
$\varphi(x)\sim x$ and
\item\label{levb}
the superorbit $\sorb(x)$ is a level set of $\fix\T$.
\end{enumerate}
\end{proposition}

\begin{proof}
The proof of \ref{inva} is clear. 
For \ref{levb}, it suffices to show that $\aorb (x)$ is a level set of $\fix\T$ for all $x\in K$. For $x\in K$, $f\in\fix\T$ and $\varepsilon>0$ define the closed neighborhood ${U:=\{y\in K\colon |f(y)-f(x)|\le\varepsilon\}}$ of $x$ which is $\varphi$-invariant because $f$ is a fixed function of $\T$. This implies $\aorb(x)\subseteq U$ by the definition of the approximating orbit. If $z\in\aorb(x)$, then $z\in U$, hence $|f(z)-f(x)|\le\varepsilon$ for each $\varepsilon > 0$ showing $f(z)=f(x)$.
\end{proof}

We now give examples for approximating orbits, respectively, superorbits and analyze the corresponding quotient space.

\begin{example}
\label{bspx2}
Take $K:=[0,1]$ and $\varphi(x):=x^2$ for $x\in K$. Then \[\aorb (x)=\begin{cases}
\orb(x)& \text{for }x\in [0,1),\\
[0,1] & \text{for }x=1.
\end{cases}\] 
Hence \[\aorb(1)=K\] inducing the trivial decomposition of $K$. The corresponding quotient space is a singleton and therefore Hausdorff, hence corresponds to the fixed factor $L$ by Proposition \ref{technik}. This is in accordance with $\dim\fix\T=1$.
\end{example}

\begin{example}
\label{niedex}
Take the compact space 
\[K:=\left\{(c,0)\colon c\in [0,1] \right\} \,\dot\cup\,\left\{ \left(\tfrac{k}{n},\tfrac{1}{n}\right)\colon n\in \mathbb{N},\, k=0,...,n-1\right\}\subseteq \mathbb{R}^2\] 
and consider on $K$ the continuous dynamics
\[\varphi(x):=
\begin{cases}
x &\text{if }x=\left(c,0\right)\text{ for some }c\in [0,1],\\
\left(\frac{k+1}{n},\frac{1}{n}\right) & \text{if }x=\left(\frac{k}{n},\frac{1}{n}\right)\text{ for some }\,n\in\mathbb{N}\text{ and }k\in\{0,...,n-2\},\\
\left(\frac{n-1}{n},0\right) & \text{if }x=\left(\frac{n-1}{n},\frac{1}{n}\right)\text{ for some }\,n\in\mathbb{N}.
\end{cases}
\]

\newcommand{\Anzahl}{23}

\begin{figure}[H]
\begin{tikzpicture}
\def\num{\Anzahl} 

\def\scal{5.} 
\filldraw (0,\scal) circle (1pt); 

%
\node[below,left] at (0, 0) {0};
\node[left] at (0, \scal) {1};
\node[below,right] at (\scal, 0) {1};
\draw [thick] (0,0)-- (\scal,0);

\def\sx{1.}
\def\bb{1}
\def\vax{1.}
\def\vay{1.}
\def\vbx{1.}
\def\vby{1.}

\pgfmathparse{1.}     \let \mx=\pgfmathresult
\pgfmathparse{\mx}    \let \valx=\pgfmathresult

\pgfmathparse{\mx-1.} \let \valy=\pgfmathresult

\pgfmathparse{1./\valx }        \let \sx=\pgfmathresult

\pgfmathparse{\scal*(\valy    )*\sx } \let \vax=\pgfmathresult
\pgfmathparse{\scal*\sx }             \let \vay=\pgfmathresult
\pgfmathparse{\scal*(\valy    )*\sx } \let \vbx=\pgfmathresult
\pgfmathparse{\scal* 0.       }         \let \vby=\pgfmathresult

\draw[->,dotted,thick] (\vax,\vay) to (\vbx,\vby);

\foreach \mx in {2,...,\num}{
\pgfmathparse{\mx} \let \valx=\pgfmathresult
\pgfmathparse{\mx-1} \let \bb=\pgfmathresult 

\foreach \my in {1,...,\bb}{
\pgfmathparse{\my} \let \valy=\pgfmathresult

\pgfmathparse{1. /\valx }         \let \sx=\pgfmathresult

\pgfmathparse{\scal*(\valy-1.0)*\sx }  \let \vax=\pgfmathresult
\pgfmathparse{\scal*\sx }              \let \vay=\pgfmathresult
\pgfmathparse{\scal*(\valy-0.01)*\sx } \let \vbx=\pgfmathresult
\pgfmathparse{\scal*\sx }              \let \vby=\pgfmathresult

\filldraw (\vbx,\vby) circle (1pt); 
\draw[->,dotted,thick] (\vax,\vay) to (\vbx,\vby);
}

\pgfmathparse{\mx-1.} \let \valy=\pgfmathresult

\pgfmathparse{1. /\valx }        \let \sx=\pgfmathresult

\pgfmathparse{\scal*(\valy    )*\sx } \let \vax=\pgfmathresult
\pgfmathparse{\scal*\sx }             \let \vay=\pgfmathresult
\pgfmathparse{\scal*(\valy    )*\sx } \let \vbx=\pgfmathresult
\pgfmathparse{\scal* 0.       }         \let \vby=\pgfmathresult

\draw[->,dotted,thick] (\vax,\vay) to (\vbx,\vby);
}
\end{tikzpicture}
\end{figure}

\begin{enumerate}
\item
The approximating orbits are
\[\aorb(x)=
\begin{cases}
\{(a,0)\colon a\in [c,1]\} &\text{if }x=(c,0)\text{ for some }c\in [0,1],\\
\mathrm{orb}(x)=\left\{\left(\frac{k+m}{n},\frac{1}{n}\right)\colon m=0,...,n-k-1\right\}\cup\left\{\frac{n-1}{n},0\right\}& \text{if }x=\left(\frac{k}{n},\frac{1}{n}\right)\text{ for some }\,n\in\mathbb{N}\\&\text{and }\,k\in\{0,...,n-1\}.
\end{cases}
\]
\item
If $x=\left(\frac{k}{n},\frac{1}{n}\right)\in K$ for some $n\in \mathbb{N}$ and $k\in\{0,...,n-2\}$, we have \[\aorb (x)\cap\aorb\left(\tfrac{n-1}{n},0\right)= \left(\tfrac{n-1}{n},0\right)\neq\emptyset.\] This implies \[\aorb\left(\tfrac{n-1}{n},0\right)\cap\aorb(x_1,0)=\left\{(c,0)\colon c\in[\tfrac{n-1}{n},1]\}\cap\{(c,0)\colon c\in [x_1,1]\right\}\neq\emptyset\] for all $x_1\in [0,1]$. Hence $x\sim y$ for all $y\in K$ yielding \[\sorb(x)=K.\] 
\end{enumerate}
Therefore, the quotient space induced by the superorbits is a singleton and hence a Hausdorff space, thus corresponds to the one-dimensional fixed space of $\T$. 
\end{example}

While the above superorbits were sufficient to characterize the fixed space of $\T$, the next example reveals that this is not always the case.

\begin{example}
\label{infty0}

Let $K:=[0,\infty]$ be the one-point compactification of $[0,\infty)$ and \[\varphi\colon K\to K,\quad x\mapsto 
\begin{cases}
(x-n)^{2}+n & \text{for }x\in [n,n+1),\,n\in\mathbb{N}_0, \\
\infty & \text{for } x=\infty
\end{cases}\] 
(see Figure \ref{x2unendl}).
\begin{figure}[h]
\caption{$(K;\varphi)$}
\begin{center}
\captionsetup{labelformat=empty}
\label{x2unendl}
\includegraphics[width=0.4\textwidth]{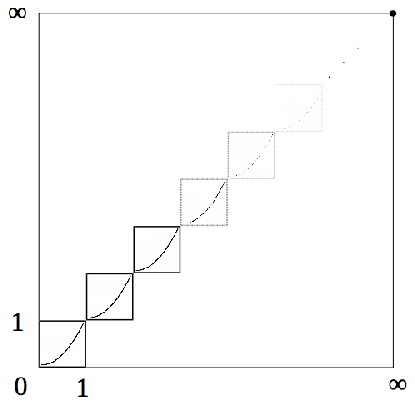}
\end{center}
\end{figure}

Then the approximating orbits are \[\aorb (x)=\begin{cases}
\{0\}&\text{for }x=0,\\
[n-1,n]& \text{for }x=n\in\mathbb{N},\\
\orb(x)=\left\{(x-n)^{2k}-n:\,k\in\mathbb{N}_0\right\}\cup\{n\}&\text{for }x\in (n,n+1),\,n\in\mathbb{N}_0,\\
\{\infty\}& \text{for }x=\infty.
\end{cases}\] 
This yields the superorbits
\[\sorb (x)=\begin{cases}
[0,\infty)& \text{for }0\le x<\infty,\\
\{\infty\} & \text{for }x=\infty.
\end{cases}\] 
However, since $\dim\fix T_{\varphi}=1$, the maximal level set of $\fix\T$ is $[0,\infty]$. Hence the quotient space induced by the superorbits is not Hausdorff.
\end{example}

\subsection{Superorbits of finite degree}
To obtain a Hausdorff quotient space, we iterate the process of building intersections of certain neighborhoods (approximating orbits) and then defining an equivalence relation yielding superorbits. 

\begin{definition}
\label{defendl}
Let $n\in\mathbb{N}_0$ and $x\in K$.
\begin{enumerate}
\item[]\textbf{Base case:}

For $n=0$, define the \emph{approximating orbit of $x$ of degree $0$} as
\[\aorb_0(x):=\aorb(x)\] and the \emph{superorbit  of $x$ of degree $0$} as
\[\sorb_0(x):=\sorb(x)\]
as in Definition \ref{aorbsorb}.
\item[]
\begin{samepage}
\textbf{Successor case:}
\label{succ}
\begin{enumerate}[(i)]
\item Let $n\ge 1$.
The \emph{approximating orbit of $x$ of degree $n$} is
\[\aorb_{n}(x):=
\bigcap\limits_{\substack{U\in \mathcal{U}(x)\text{ open},\\\sorb_{n-1}(U)\subseteq U}}\overline{U}^{\sorb_{n-1}}
\]
with $\sorb_{n-1}(U):=\bigcup_{y\in U}\sorb_{n-1}(y)$ for $U\subseteq K$ and \[\overline{U}^{\sorb_{n-1}}:=\bigcap\limits_{\substack{F\in\mathcal{U}(x)\text{ closed},\\U\subseteq F,\\\sorb_{n-1}(F)\subseteq F}} F,\] called the \emph{$\sorb_{n-1}$-closure of $U$}.
\item 
Let $\sim_{n}$ be the equivalence relation generated by $(\aorb_n(x))_{x\in K}$ with canonical projection $\pi_n\colon K\to K/{\sim_n}$. The \emph{superorbit of $x$ of degree $n$} is
\[\sorb_{n} (x):=\pi_n^{-1}([x])=\bigcup\limits_{\substack{y\in K,\\ y\sim_{n} x}}\aorb_{n}(y).\] 
\end{enumerate}
\end{samepage}
\end{enumerate}
\end{definition}

As before, we check the assumptions \ref{inv} and \ref{lev} in Proposition \ref{technik}.

\begin{proposition}
\label{levelfinite}
For each $n\in\mathbb{N}_0$ and $x\in K$ we have that
\begin{enumerate}
\item
$\varphi(x)\sim_n x$ and
\item
the superorbit $\sorb_n(x)$ of degree $n$ is a level set of $\fix\T$.
\end{enumerate}
\end{proposition}

\begin{proof}
\begin{enumerate}
\item
It suffices to show that $\sorb_n(x)$ is $\varphi$-invariant for all $n\in \mathbb{N}_0$ and $x\in K$. For $n=0$, see Proposition \ref{lemlev}. If $\sorb_n(x)$ is $\varphi$-invariant for all $x\in K$ and $n\in\mathbb{N}_0$, then also $\sorb_{n+1}(x)$ is $\varphi$-invariant by $\sorb_n(y)\subseteq\aorb_{n+1}(y)\subseteq \sorb_{n+1}(x)$ for $x,y\in K$ with $y\in\sorb_{n+1}(x)$.
\item
For $n=0$ see Proposition \ref{lemlev}. For $n\in\mathbb{N}_0$ and $x\in K$ assume that $\sorb_n(x)$ is a level set of $\T$. We show that the assertion holds true for $n +1$. 
As in the base case, consider $U:=\{y\in K:\,|f(y)-f(x)|\le\varepsilon\}=f^{-1}(\overline{B_\varepsilon(f(x))})$ for some $f\in\fix\T$ and $\varepsilon>0$. To prove $\aorb_{n+1}(x)\subseteq U$, we construct some open $V\in \mathcal{U}(x)$ with $\sorb_n (V)\subseteq V$ such that $\overline{V}^{\sorb_n}\subseteq U$.

Define \[V:=f^{-1}(B_\varepsilon(f(x))).\] Then $V\in \mathcal{U}(x)$, $V$ is open and $V\subseteq U$. By the induction hypothesis we have for $x'\in K$ with $x\sim_n x'$ that $f(x)=f(x')$. Hence by the universal property of the quotient topology there is some unique continuous function $\hat{f}\colon K/{\sim_n}\to \mathbb{C}$ such that $f=\hat{f}\circ\pi_n$, i.e., the following diagram commutes.

\hfil
\begin{xy}
  \xymatrix{
      K \ar[r]^f \ar[d]_{\pi_n}   &   \mathbb{C}  \\
      K/{\sim_n} \ar[ru]_{\hat{f}}          
  }
\end{xy}\\\\
\hfil

\noindent

This implies $V=\pi_n^{-1}\left(\hat{f}^{-1}\left(B_\varepsilon\left(f\left(x\right)\right)\right)\right)$, hence

\begin{equation}
\label{sorbinv}
\begin{aligned}
\sorb_n(V)&=\pi_n^{-1}(\pi_n(V))=V.
\end{aligned}
\end{equation}

This yields $\aorb_{n +1}(x)\subseteq \overline{V}^{\sorb_n}$.

We now show that $\overline{V}^{\sorb_n}\subseteq U$. For $f\in\fix\T$ and $C\subseteq \mathbb{C}$, we have $\sorb_n(f^{-1}(C))=f^{-1}(C)$ by the universal property of the quotient topology as above. Moreover, if $B_\varepsilon(f(x))\subseteq C$ then $V\subseteq f^{-1}(C)$ by definition. 

By this,
\begin{eqnarray*}
U&=&f^{-1}(\overline{B_\varepsilon(f(x))})=\bigcap\limits_{\substack{C\subseteq\mathbb{C}\text{ closed},\\B_\varepsilon(f(x))\subseteq C}}f^{-1}(C)\\&\supseteq& \bigcap\limits_{\substack{C\subseteq\mathbb{C},\\f^{-1}(C)\text{ closed},\\f^{-1}(B_\varepsilon(f(x)))\subseteq f^{-1}(C)}}f^{-1}(C)\supseteq\bigcap\limits_{\substack{F\subseteq K\text{ closed},\\\sorb_n(F)\subseteq F,\\ V\subseteq F}}F\\&=&\overline{V}^{\sorb_n}.
\end{eqnarray*}

Hence $\aorb_{n +1}(x)\subseteq \overline{V}^{\sorb_n}\subseteq U$. This implies that for $z\in\aorb_{n +1}(x)$, also $z\in U$. Hence $|f(z)-f(x)|\le\varepsilon$ by definition of $U$. Since $\varepsilon$ is arbitrary, this implies $f(z)=f(x)$.

\end{enumerate}
\end{proof}

We now give a concrete example for these new orbits and analyze the corresponding quotient space.
\begin{example}
\label{infty1}
\begin{enumerate}
\item
Let $K:=[0,\infty]$ be the one-point compactification of $[0,\infty)$ and $\varphi_1\colon K\to K$ with $\varphi_1:=\varphi$ as in Example \ref{infty0}.
As seen before, $\dim\fix T_{\varphi_1}=1$ and  \[\sorb_0 (x)=
\begin{cases}
[0,\infty)& \text{for }0\le x<\infty,\\
\{\infty\} & \text{for }x=\infty.
\end{cases}\] 
Since $[0,\infty)$ is the only $\sorb_0$-invariant open subset of $K$, we have for all $x\in K$ \[\aorb_1(x)=\sorb_1(x)=K.\]

\begin{figure}
\begin{center}
\caption{$(\tilde{K};\tilde{\varphi}_1)$}
\label{phi1schlangekari}
\includegraphics[width=0.4\textwidth]{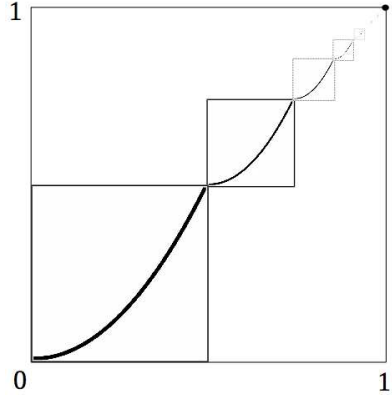}
\end{center}
\end{figure}

\noindent
Next we define an isomorphic system $(\tilde{K};\tilde{\varphi}_1)$ by $\tilde{K}:=[0,1]$ and 
 \[\tilde{\varphi}_1\colon \tilde{K}\to \tilde{K}, \quad\tilde{\varphi}_1:=h\circ\varphi_1\circ h^{-1} \] 
 for a homeomorphism $h\colon K\to \tilde{K}$ with $h(0)=0$ and $h(\infty)=1$ (see Figure \ref{phi1schlangekari}). For this ``compressed'' system we still have $\dim\fix T_{\tilde{\varphi}_1}=1$.

\item
Analogously, we construct a system $(K;\varphi_2)$ on $K=[0,\infty]$ with $\sorb_2(x)=K$ and $\sorb_1(x)\subsetneq K$ for all $x\in K$ via
\[\varphi_2\colon  K\to K,\quad x\mapsto 
\begin{cases}
\tilde{\varphi}_1(x-m)+m & \text{for }x\in [m,m+1),\,m\in\mathbb{N}_0, \\
\infty & \text{for } x=\infty.
\end{cases}
\]
We iterate this procedure of compressing systems and lining up copies of these on $K=[0,\infty]$ (cf. Figure \ref{iter}). By this procedure we obtain systems $(K;\varphi_n)$ with $\sorb_{n-1}(x)\subsetneq K$ and $\sorb_n(x)=K$ for some $n\in\mathbb{N}$ and all $x\in K$. Hence the quotient space $K/{\sim_n}$ is a singleton, thus homeomorphic to the fixed factor $L$ by Proposition \ref{technik}.
\begin{figure}
\begin{center}
\caption{Construction of systems with $\sorb_n(x)=K$ for $x\in K$, $n\in\mathbb{N}$}
\label{iter}
\includegraphics[width=0.7\textwidth]{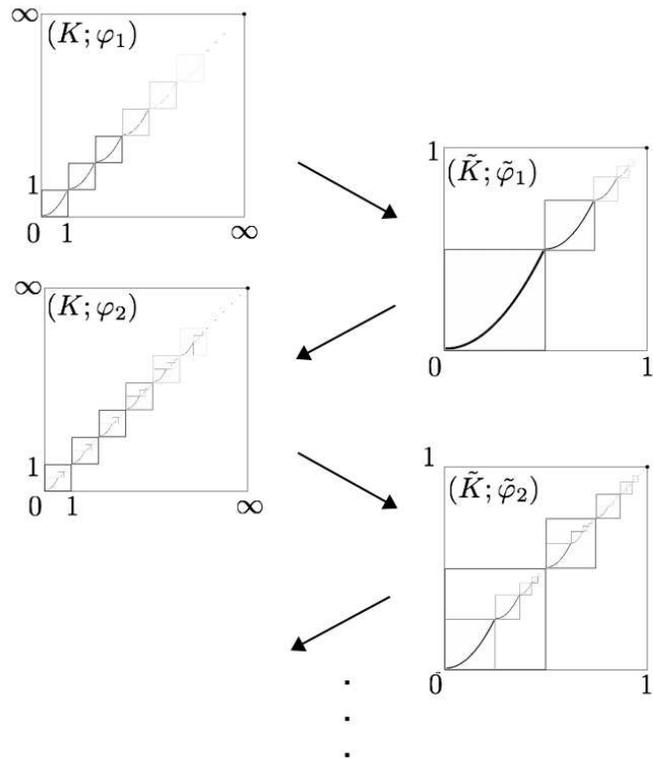}
\end{center}
\end{figure}
\end{enumerate}
\end{example}

This construction leads to an example in which even superorbits of arbitrary degree $n\in\mathbb{N}$ are not sufficient to characterize the fixed space.

\begin{example}
\label{exomega}

Let $K:=[0,\infty]$ and define
\[\varphi(x):=\tilde{\varphi}_k(x)\] for $x\in [k-1,k)$, $k\in\mathbb{N}$, with $\tilde{\varphi}_k$ as in Example \ref{infty1}. 
\begin{figure}[H]
\begin{center}
\caption{$(K;\varphi)$}
\includegraphics[width=0.4\textwidth]{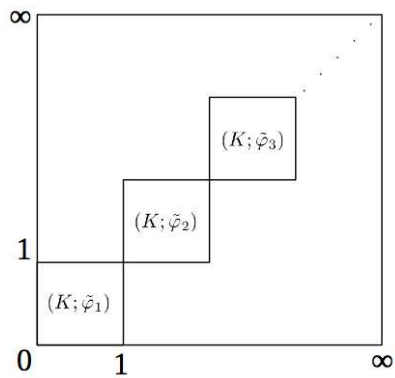}
\end{center}
\label{bspnonfin}
\end{figure}
Then for $n\in\mathbb{N}_0$ the superorbit of degree $n$ is
\[\sorb_{n}(x)=
\begin{cases}
[0,n+1)&\text{for }x\in [0,n+1),\\
\{\infty\}&\text{for }x=\infty
\end{cases}
\]
and \[\sorb_n(x)\subseteq [n+1,\infty)\text{ for }x\in[n+1,\infty).\]
Hence $\sorb_n(x)\neq K$ for all $x\in K$ and $n\in\mathbb{N}_0$ which implies that the corresponding quotient space $K/{\sim_n}$ contains more than one element. Thus, it does not correspond to the fixed factor $L$, which is a singleton since $\dim\fix\T=1$. 
\end{example}

\subsection{Superorbits of non-finite degree}
Because superorbits of arbitrary finite degree do not, in general, yield a Hausdorff quotient space, we introduce superorbits of non-finite degree using the theory of ordinal numbers. We propose the following definition, where $\omega$ denotes the first non-finite ordinal number (see, e.g., \cite[p. 42, Definition 6.1]{dugun}).

\begin{definition}
\label{defunendl}
For any $x\in K$ the \emph{approximating orbit of $x$ of degree $\omega$} is
\[
\aorb_{\omega}(x):=\bigcup\limits_{n\in \mathbb{N}}\sorb_n (x),
\]
and the \emph{superorbit of $x$ of degree $\omega$} is
\[
\sorb_{\omega}(x):=\pi_\omega^{-1}([x])=\bigcup\limits_{\substack{y\in K,\\y\sim_\omega x}}\aorb_\omega (y),
\]
where $\sim_\omega$ is the equivalence relation generated by $(\aorb_\omega(x))_{x\in K}$ with canonical projection $\pi_\omega\colon K\to K/{\sim_\omega}$.
\end{definition}

By Proposition \ref{levelfinite}, the assumptions \ref{inv} and \ref{lev} of Proposition \ref{technik} are satisfied for this equivalence relation.

\begin{proposition}
For each $x\in K$ we have that
\begin{enumerate}
\item
$\varphi(x)\sim_\omega x$ and
\item
the superorbit $\sorb_\omega(x)$ of degree $\omega$ is a level set of $\fix\T$.
\end{enumerate}
\end{proposition}

Even superorbits of degree $\omega$ do not, in general, yield a Hausdorff quotient space as the following example shows.

\begin{example}
Let again $K:=[0,\infty]$ be the one-point compactification of $[0,\infty)$ and $\tilde{K}:=[0,1]$. Consider the system $(K;\varphi)$ as in Example \ref{exomega} and the isomorphic system  \[\tilde{\varphi}\colon \tilde{K}\to \tilde{K}, \quad\tilde{\varphi}:=h\circ\varphi\circ h^{-1} \] 
 for a homeomorphism $h\colon K\to \tilde{K}$ with $h(0)=0$ and $h(\infty)=1$. 
Analogously to Example \ref{infty1}, we construct a system $(K;\psi)$ via
\[\psi\colon  K\to K,\quad x\mapsto 
\begin{cases}
\tilde{\varphi}(x-n)+n & \text{for }x\in [n,n+1),\,n\in\mathbb{N}_0, \\
\infty & \text{for } x=\infty,
\end{cases}
\]
by putting copies of the compressed system in a row. Then the fixed factor $L$ is a singleton, while $\sorb_\omega (x)\neq K$ for all $x\in K$, thus the corresponding quotient space $K/{\sim_\omega}$ contains more that one point.
\end{example}

\section{Characterization of the fixed space via transfinite superorbits}


\label{chap4}

To achieve our goal, we need superorbits for arbitrary ordinal numbers. We define the base case, successor case and limit case analogously to Definitions \ref{defendl} and \ref{defunendl}. The class of ordinal numbers is denoted by $\mathrm{Ord}$. 

\begin{definition}
Let $x\in K$.
\begin{enumerate}
\item[]
\textbf{Base case:}
\begin{enumerate}[(i)]
\item
The \emph{approximating orbit of $x$ of degree 0} is
\[\aorb_0 (x):=\bigcap\limits_{\substack{U\in \mathcal{U}(x) \text{ closed},\\\varphi(U)\subseteq U }}U.\] 
\item
Let $\sim_0$ be the equivalence relation on $K$ generated by $(\aorb_0(x))_{x\in K}$ with canonical projection $\pi_0:K\to K/{\sim_0}$. The \emph{superorbit of $x$ of degree 0} is
\[\sorb_0 (x):=\pi^{-1}_0([x])=\bigcup\limits_{\substack{y\in K,\\ y\sim_0 x}}\aorb_0(y).\]

\end{enumerate}
\begin{samepage}
\item[]\textbf{Successor case:}

\begin{enumerate}[(i)]
\item
Let $0\neq\gamma\in\mathrm{Ord}$ a successor. 
Then the \emph{approximating orbit of degree $\gamma$} is
\[\aorb_{\gamma}(x):=
\bigcap\limits_{\substack{U\in \mathcal{U}(x)\text{ open},\\\sorb_{\gamma-1}(U)\subseteq U}}\overline{U}^{\sorb_{\gamma-1}}
\]

where $\sorb_{\gamma-1}(U):=\bigcup_{y\in U}\sorb_{\gamma-1}(y)$ and 
\begin{equation}
\label{sorbclos}
\overline{U}^{\sorb_{\gamma-1}}:=\bigcap\limits_{\substack{F\in\mathcal{U}(x)\text{ closed},\\U\subseteq F,\\\sorb_{\gamma-1}(F)\subseteq F}} F
\end{equation} 

denotes the \emph{$\sorb_{\gamma-1}$-closure of $U$}.

\item 
As before let $\sim_{\gamma}$ be the equivalence relation on $K$ generated by $(\aorb_\gamma(x))_{x\in K}$ and denote the canonical projection by $\pi_\gamma\colon K\to K/{\sim_\gamma}$. Finally, the \emph{superorbit of $x$ of degree $\gamma$} is
\[\sorb_{\gamma} (x):=\pi_\gamma^{-1}([x])=\bigcup\limits_{\substack{y\in K,\\ y\sim_{\gamma} x}}\aorb_{\gamma}(y).\] 
\end{enumerate}

\end{samepage}

\item[]\textbf{Limit case:}

Let $0\neq\gamma\in\mathrm{Ord}$ be a limit ordinal. 
Then the \emph{approximating orbit of $x$ of degree $\gamma$} is
\[
\aorb_{\gamma}(x):=\bigcup\limits_{\beta< \gamma}\sorb_{\beta} (x).
\]
The equivalence relation $\sim_\gamma$ on $K$ and the superorbit $\sorb_\gamma (x)$ of degree $\gamma$ are defined as in the successor case.
\end{enumerate}
\end{definition}

%
%

Before proving that superorbits of arbitrary degree are level sets, we list some basic properties.
\begin{proposition}
\label{propsorb}
Let $\gamma\in\mathrm{Ord}$ and $x\in K$. 
\begin{enumerate}
\item \label{invariance}
For all $\beta\le\gamma$ we have \[\sorb_\beta(x)\subseteq \sorb_{\gamma}(x)\] and \[\aorb_\beta(x)\subseteq \aorb_{\gamma}(x).\]
\item
\label{lemma1}
For $S\subseteq K$ the following are equivalent.
\begin{enumerate}[(i)]
\item \label{1} $\sorb_\gamma(S)\subseteq S$.
\item \label{2} $\sorb_\gamma(S)=S$.
\item \label{4} There is some $M\subseteq K$ such that $S=\bigcup\limits_{y\in M}\sorb_\gamma(y)$.
\end{enumerate}
\end{enumerate}
\end{proposition}
\begin{proof}	
\begin{enumerate}
\item Let $x\in K$. We first show $\sorb_\beta(x)\subseteq \sorb_{\gamma}(x)$ for all $\beta\le\gamma$ using transfinite induction. 
For $\gamma=0$ the statement is trivial. For $\gamma\in\mathrm{Ord}$ assume $\sorb_\beta(x)\subseteq \sorb_{\gamma}(x)$ for all $\beta\le\gamma$. We show $\sorb_\beta(x)\subseteq \sorb_{\gamma+1}(x)$ for all $\beta\le\gamma+1$. This follows immediately from  \[\sorb_\gamma(x)\subseteq \aorb_{\gamma+1}(x)\subseteq \sorb_{\gamma+1}(x).\] Now let $\gamma$ be a limit and $\beta\le\gamma$. Then $\sorb_\beta(x)\subseteq \sorb_{\gamma}(x)$ by the definition of $\sorb_{\gamma}(x)$.\\

Similarly, we show $\aorb_\beta(x)\subseteq \aorb_{\gamma}(x)$ for all $\beta\le\gamma$. 
\item 
The implications $\ref{1}\Leftrightarrow \ref{2}\Rightarrow \ref{4}$ are trivial.
\noindent
%
\noindent
$\ref{4}\Rightarrow \ref{2}$: 
By assumption we have $S\subseteq \bigcup\limits_{y\in M}\sorb_\gamma(y)$ and thus for all $z\in S$ there is some $y\in M$ such that $z\in \sorb_\gamma(y)$. Since $\sim_\gamma$ is an equivalence relation, we have $\sorb_\gamma(z)=\sorb_\gamma(y)$, hence $\underbrace{\bigcup\limits_{y\in S}\sorb_\gamma(y)}_{=\sorb_\gamma(S)}\subseteq \bigcup\limits_{y\in M}\sorb_\gamma(y)$. The inclusion $\bigcup\limits_{y\in M}\sorb_\gamma(y)=S\subseteq \bigcup\limits_{y\in S}\sorb_\gamma(y)=\sorb_\gamma(S)$ clearly holds.
\end{enumerate}
\end{proof}

We now show that the properties \ref{inv} and \ref{lev} in Proposition \ref{technik} hold for all superorbits.
%
%

\begin{proposition}
For each $\gamma\in\mathrm{Ord}$ and $x\in K$ we have that
\begin{enumerate}
\item
$\varphi(x)\sim_\gamma x$ and
\item
the superorbit $\sorb_\gamma(x)$ of degree $\gamma$ is a level set of $\fix\T$.
\end{enumerate}
\end{proposition}

\begin{proof}
We use transfinite induction. For the base case $\gamma=0$ see Proposition \ref{lemlev}. If $\gamma\in \mathrm{Ord}$ is a successor, the proof works analogously to Proposition \ref{levelfinite}. Let thus $\gamma\in \mathrm{Ord}$ be a limit. Clearly, $\varphi(x)\sim_\gamma x$ for all $x\in K$. Assume that $\sorb_\beta(x)$ is a level set of $\T$ for all $x\in K$, $\beta\le\gamma$. Then $\sorb_\gamma(x)$ is a level set of $\T$ by definition and Proposition \ref{propsorb}.\ref{invariance}.
\end{proof}

The next proposition is crucial for the proof of our Main Theorem \ref{charfix}. It shows that an approximating orbit corresponds to the intersection of closed neighborhoods in the quotient space. 
%


\begin{proposition}
\label{prop}
If $x\in K$, $\gamma\in \mathrm{Ord}$ a successor and  $\pi_\gamma:K\to K/{\sim_\gamma}$ the canonical projection, then \[\pi_\gamma\left(\aorb_{\gamma+1}(x)\right)= \bigcap\limits_{\substack{U_\sim\in\mathcal{U}([x]_\gamma)\\ \text{closed}}}U_\sim\] and \[\pi_{\gamma}^{-1}\left(\bigcap\limits_{\substack{U_\sim\in\mathcal{U}([x]_\gamma)\\ \text{closed}}}U_\sim\right)=\aorb_{\gamma+1}(x).\]
\end{proposition}

\begin{proof}
Since $\pi_\gamma$ is surjective, it suffices to show the following inclusions:

\begin{enumerate}
\item \label{bild}
$\pi_\gamma(\aorb_{\gamma+1}(x))\subseteq\bigcap\limits_{\substack{U_\sim\in\mathcal{U}([x]_\gamma)\\ \text{ closed}}}U_\sim$,
\item \label{urbild}
$\pi_{\gamma}^{-1}\left(\bigcap\limits_{\substack{U_\sim\in\mathcal{U}([x]_\gamma) \\ \text{closed}}}U_\sim\right)\subseteq\aorb_{\gamma+1}(x)$.
\end{enumerate}

By the definition of an approximating orbit and the results obtained in Proposition \ref{propsorb} and Lemma \ref{propsorb} \ref{lemma1}, we deduce
\begin{eqnarray*}
\pi_\gamma\left(\aorb_{\gamma+1}(x)\right)
&{\overset{\ref{propsorb}\, \ref{lemma1}}{=}}&\pi_\gamma\left(\bigcap\limits_{\substack{U\in\mathcal{U}(x) \text{ open},\\\sorb_{\gamma}(U)= U}}\overline{U}^{\sorb_\gamma}\right)\\
&\underset{\sorb_\gamma(\overline{U}^{\sorb_\gamma})=\overline{U}^{\sorb_\gamma}}{\overset{\ref{remeqcl}}{=}}&\pi_\gamma\left(\bigcap\limits_{\substack{U\in\mathcal{U}(x) \text{ open},\\\sorb_{\gamma}(U)= U}}\pi^{-1}_{\gamma}\left(\pi_{\gamma}(\overline{U}^{\sorb_\gamma})\right)\right)\\
&=&\bigcap\limits_{\substack{U\in\mathcal{U}(x) \text{ open},\\\pi^{-1}_{\gamma}(\pi_{\gamma}(U))= U}}\pi_{\gamma}(\overline{U}^{\sorb_\gamma})\\
&\subseteq & \bigcap\limits_{U_\sim\in\mathcal{U}([x]_\gamma) \text{ closed}}U_\sim
\end{eqnarray*}
which proves \ref{bild}.

To show \ref{urbild}, let $[z]_\gamma\in\bigcap\limits_{U_\sim\in\mathcal{U}([x]_\gamma) \text{ closed}}U_\sim$. Since $\sorb_\gamma(z)=\pi_\gamma^{-1}([z]_\gamma)$, we show $\sorb_\gamma(z)\subseteq\aorb_{\gamma+1}(x)$. By the definition of $\aorb_{\gamma+1}(x)$ it suffices to show $\sorb_\gamma(z)\subseteq \overline{U}^{\sorb_\gamma}$ for $U\in\mathcal{U}(x)$ open with $\sorb_\gamma (U)\subseteq U$.
%
%
%
%
%
%
%

We now move to the quotient space. Define \[V_\sim:=\pi_\gamma(\overline{U}^{\sorb_\gamma}):=\left\{\pi_\gamma(y):\,y\in \overline{U}^{\sorb_\gamma}\right\}.\] To show $V_\sim\in\mathcal{U}([x]_\gamma)$, we check the following.
\begin{enumerate}
\item[(i)] $[x]_\gamma\in V_\sim$ and 
\item[(ii)] there is some subset $W_\sim\subseteq V_\sim$ which is open in $K/{\sim_\gamma}$ and $[x]_\gamma\in W_\sim$.
\end{enumerate}
We have $\pi_\gamma^{-1}([x]_\gamma)=\sorb_\gamma(x)\subseteq \overline{U}^{\sorb_\gamma}$ since $x\in \overline{U}^{\sorb_\gamma}$ and $\sorb_\gamma (\overline{U}^{\sorb_\gamma})\subseteq \overline{U}^{\sorb_\gamma}$. Therefore, $\left\{[x]_\gamma\right\}=\pi_\gamma(\pi_\gamma^{-1}([x]_\gamma))\subseteq\pi_\gamma(\overline{U}^{\sorb_\gamma})=V_\sim$, hence $[x]_\gamma\in V_\sim$ showing $(i)$.

Define \[W_\sim:=\pi_\gamma(U).\] Then $U\subseteq \overline{U}^{\sorb_\gamma}$ implies $W_\sim=\pi_\gamma(U)\subseteq \pi_\gamma(\overline{U}^{\sorb_\gamma})= V_\sim$. 

Furthermore, $W_\sim$ is open with respect to the quotient topology since $\pi_\gamma^{-1}(W_\sim)=\pi_\gamma^{-1}(\pi_\gamma(U))
=\sorb_\gamma (U)\overset{\ref{propsorb}\,\ref{lemma1}}{=}U$ is open. Clearly, $[x]_\gamma\in W_\sim$. This shows $(ii)$.
\\\\
Analogously, we see $\pi_\gamma^{-1}(V_\sim)=\overline{U}^{\sorb_\gamma}$. Hence $V_\sim$ is closed with respect to the quotient topology.
\\\\
Summarizing, we obtain $V_\sim\in\mathcal{U}([x]_\gamma)$ and $V_\sim$ closed, hence $[z]_\gamma\in V_\sim$ by assumption. This implies $\sorb_\gamma(z)=\pi_\gamma^{-1}([z]_\gamma)\subseteq \pi_\gamma^{-1}(V_\sim)=\overline{U}^{\sorb_\gamma}$ which proves assertion \ref{urbild}.

\end{proof}

To obtain a Hausdorff quotient space corresponding to the fixed factor $L$, the process of building superorbits must become stationary.

\begin{theorem}
\label{stationary}
There is some ordinal number $\gamma\in\mathrm{Ord}$ such that $\sorb_\gamma (x)=\sorb_{\gamma+1}(x)$ for all $x\in K$. 
\end{theorem}

\begin{proof}
For all $\beta\in \mathrm{Ord}$ we have $|\left\{\sorb_\alpha(x):\,x\in K,\alpha\le \beta\right\}|\le|\mathfrak{P}(K)|$ for the power set $\mathfrak{P}(K)$ of $K$. Moreover, by Proposition \ref{propsorb}.\ref{invariance}, if $\sorb_\alpha(x)=\sorb_{\alpha+1}(x)$ for some $\alpha\in \mathrm{Ord}$ and some $x\in K$, then $\sorb_\alpha(x)=\sorb_{\alpha'}(x)$ for all $\alpha'\in \mathrm{Ord}$ with $\alpha\le\alpha'$. This implies for $\gamma\in \mathrm{Ord}$ with $|\gamma|=|\mathfrak{P}(K)|$ that $\sorb_\gamma(x)=\sorb_{\gamma+1}(x)$ for all $x\in K$.
\end{proof}

We can now describe the fixed space of $\T$ in terms of $(K;\varphi)$.

\begin{maintheorem}
\label{charfix}
Let $\fix\T\cong C(L)$ for a compact space $L$. Then $L$ is homeomorphic to $ K/{\sim_\alpha}$ for some $\alpha\in\mathrm{Ord}$. 
\end{maintheorem}

\begin{proof}
Choose $\alpha\in\mathrm{Ord}$ such that $\sorb_\alpha(x)=\sorb_{\alpha+1}(x)$ for all $x\in K$ (see Theorem \ref{stationary}) and assume, without loss of generality, that $\alpha$ is a successor. By Proposition \ref{levelsets} and Theorem \ref{technik} it remains to show that $K/{\sim_\alpha}$ is Hausdorff, i.e., \[\left\{[x]_\alpha\right\}=\bigcap\limits_{U\in\mathcal{U}([x]_\alpha)\text{ closed}}U\] for all $x\in K$ (see Lemma \ref{hausd}).
To do so, let $[z]_\alpha\in\bigcap\limits_{U\in\mathcal{U}([x]_\alpha)\text{ closed}}U$. An immediate consequence of Proposition \ref{prop} is  \[\sorb_\alpha(z)\subseteq \aorb_{\alpha+1}(x).\] Hence we have \[\sorb_\alpha(z)\subseteq\aorb_{\alpha+1}(x)\subseteq\sorb_{\alpha+1}(x)=\sorb_\alpha(x).\] This implies $\sorb_\alpha(z)\subseteq\sorb_\alpha(x)$ and hence $\sorb_\alpha(z)=\sorb_\alpha(x)$ since $\sim_\alpha$ is an equivalence relation. Therefore, also $[z]_\alpha=\pi(\sorb_\alpha(z))=\pi(\sorb_\alpha(x))=[x]_\alpha$ which shows that $K/{\sim_\alpha}$ is Hausdorff.
\end{proof}

From this, we obtain a characterization of a one-dimensional fixed space of $\T$.

\begin{definition}
We call a topological dynamical system $(K;\varphi)$ \emph{topologically ergodic} if there is some $x\in K$ and $\gamma\in \mathrm{Ord}$ such that \[K=\sorb_\gamma (x).\]
\end{definition}

\begin{remark}
\begin{enumerate}
\item
In particular, topological ergodicity is a global property depending on the dynamical behavior of $\phi$ on the entire space $K$. 
\item
Compare, e.g., \cite[p. 2144]{peterart}, \cite[p. 151]{peterbuch} or \cite[p. 31]{vries} for a different meaning of this notion.
\end{enumerate}
\end{remark}

\begin{maintheorem}
\label{charonedim}
The fixed space of $\T$ is one-dimensional if and only if $(K;\varphi)$ is topologically ergodic.
\end{maintheorem}

\begin{remark}
In continuous-time dynamical systems there is a transfinite construction yielding so-called \emph{prolongations} (cf., e.g., \cite{auslprol}, \cite[chapters II.4 and VII]{bhatia} or \cite{ura}). These are -- if adapted to the discrete-time setting -- different from approximating orbits and superorbits as can be seen from Example \ref{niedex}. Here we have for the first prolongation \[\mathcal{D}_1(x):=\aorb_0(x)=
\begin{cases}
\{(c,0)\colon c\in [a,1]\} &\text{if }x=(a,0)\text{ for some }a\in [0,1],\\
\orb(x)&\text{elsewhere}
\end{cases}\]
and for the second prolongation 
\[\mathcal{D}_2(x):=\bigcap\limits_{U\in\mathcal{U}(x)}\overline{\bigcup\limits_{n\in\mathbb{N}}\mathcal{D}_1^n(U)}=\bigcap\limits_{U\in \mathcal{U}(x)}\mathcal{D}_1(U)=\mathcal{D}_1(x)\] for all $x\in K$ because $\mathcal{D}_1(\mathcal{D}_1(U))=\mathcal{D}_1(U)$ and $\mathcal{D}_1(U)$ is closed for all $U\in \mathcal{U}(x)$. This implies that all prolongations of higher degree are equal to $\mathcal{D}_1(x)$ for all $x\in K$, while $\dim\fix\T=1$. Hence the decomposition induced by $\fix\T$ is not obtained by the prolongations.

Also \emph{chain prolongations} (see, e.g., \cite{ding}) are in general different from our superorbits.
\end{remark}

\section{Lyapunov stability of higher order}
It is an interesting problem to decompose a topological dynamical system $(K;\varphi)$ into disjoint, $\varphi$-invariant and ``stable'' sets. We now suggest a hierarchy of stability notions which are closely linked to the fixed space $\fix\T$ of a Koopman operator.

We first recall the following standard definition. 

\begin{definition}
A set $M\subseteq K$ is called \emph{Lyapunov stable} if it is the intersection of its $\varphi$-invariant neighborhoods, i.e., \[M=\bigcap_{\substack{U\in\mathcal{U}(M),\\ \varphi(U)\subseteq U}}U.\]
\end{definition}

The maximal level sets of $\fix\T$ are Lyapunov stable and yield a decomposition of $K$. However, it may happen that there exist finer decompositions into Lyapunov stable sets as the following example shows.

\begin{example}
\begin{samepage}
Take \[K:=\left\{\left(\tfrac{1}{n},c\right)\colon n\in\mathbb{N},\, c \in [0,1] \right\}\cup\left\{\left(0,c\right)\colon c\in [0,1]\right\}\] with the subspace topology of $\mathbb{R}^2$ and the dynamics
\[\varphi(x)=\begin{cases}
\left(\tfrac{1}{n},n\left(c-\tfrac{m}{n}\right)^2+\tfrac{m}{n}\right)&\text{for } x=\left(\tfrac{1}{n},c\right) \text{ with } c\in[\tfrac{m}{n},\tfrac{m+1}{n}]\text{ for some } n\in\mathbb{N},\,m\in\mathbb{N}_0,\\
\left(0,c\right) &\text{for }x=\left(0,c\right) \text{ with } c \in [0,1]
\end{cases}\]
for $x\in K$. 
\end{samepage}

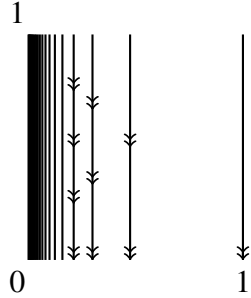
\begin{figure}[H]
\begin{tikzpicture}
\node[below] at (0, 0) {0};
\node[above] at (0, 3) {1};
\node[below] at (3, 0) {1};
\draw[thick,->>] (3,0.1)-- (3,0);
\draw[thick,->>] (1.5,1.51)-- (1.5,1.5);
\draw[thick,->>] (1.5,0.01)-- (1.5,0);
\draw[thick,->>] (1,0.01)-- (1,0);
\draw[thick,->>] (1,1.01)-- (1,1);
\draw[thick,->>] (1,2.01)-- (1,2);
\draw[thick,->>] (3/4,0.01)-- (3/4,0);
\draw[thick,->>] (3/4,3.01/4)-- (3/4,1.5/2);
\draw[thick,->>] (3/4,1.51)-- (3/4,1.5);
\draw[thick,->>] (3/4,9.01/4)-- (3/4,9/4);
\foreach \x in {1,...,20}
\draw [thick] (3/\x,3)-- (3/\x,0);
\end{tikzpicture}
\caption{$(K;\varphi)$}
\end{figure}

Here, the decomposition of $K$ induced by the fixed space $\fix\T$ is
\[K=\dot{\bigcup_{n\in\mathbb{N}}}\left\{\left(\tfrac{1}{n},c\right)\colon c \in [0,1] \right\}\,\dot{\cup}\,\left\{\left(0,c\right)\colon c \in [0,1]\right\}\]
while a finer decomposition into Lyapunov stable sets is given by
\[K=\dot{\bigcup_{n\in\mathbb{N}}}\left\{\left(\tfrac{1}{n},c\right)\colon c \in [0,1] \right\}\,\dot{\cup}\,\dot{\bigcup}_{c \in [0,1] }\left\{(0,c)\right\}.\]

\end{example}

To explain the difference between these decompositions, we use our concept of superorbits from Chapters \ref{chap3} and \ref{chap4} to generalize Lyapunov stability to a hierarchy of stability notions. This can produce decompositions of $K$ which are coarser than a decomposition into Lyapunov stable sets but finer than the decomposition induced by $\fix\T$.

\begin{definition}
\begin{enumerate}
\item
A set $M\subseteq K$ is called \emph{Lyapunov stable of degree $\alpha$} for some $\alpha\in\mathrm{Ord}$ if \[M=\bigcap\limits_{\substack{U\in\mathcal{U}(M)\text{ open},\\ \sorb_\alpha(U)\subseteq U}}\overline{U}^{\sorb_\alpha}.\]
\item 
A set $M\subseteq K$ is called \emph{absolutely Lyapunov stable} if $M$ is Lyapunov stable of degree $\alpha$ for all $\alpha\in \mathrm{Ord}$.
\end{enumerate}
\end{definition}

\begin{remark}
\label{liapincl}
If a set $M$ is Lyapunov stable of degree $\alpha$, then it is Lyapunov stable of degree $\beta$ for all $\beta \le \alpha$.
\end{remark}

\begin{samepage}
\begin{lemma}
\label{liapthm}
Let $M\subseteq K$ and $\alpha\in\mathrm{Ord}$.
\begin{enumerate}
\item
For $M$ Lyapunov stable and $x\in M$, also $\aorb_0(x)\subseteq M$ and $\sorb_0(x)\subseteq M$.
\item
For $M$ Lyapunov stable of degree $\alpha$ and $x\in M$, also $\aorb_{\alpha+1}(x)\subseteq M$ and $\sorb_{\alpha+1}(x)\subseteq M$.
\end{enumerate}
\end{lemma}
\end{samepage}

\begin{proof}
It suffices to show the assertions for the approximating orbits.
\begin{enumerate}
\item
Since $K$ is a Hausdorff space,
we have
\begin{eqnarray*}
\aorb_0(x)=\bigcap\limits_{\substack{U\in\mathcal{U}(x)\text{ closed},\\\varphi(U)\subseteq U}}U\subseteq\bigcap\limits_{\substack{U\in\mathcal{U}(M)\text{ closed},\\\varphi(U)\subseteq U}}U =\bigcap\limits_{\substack{U\in\mathcal{U}(M),\\\varphi(U)\subseteq U}}U=M.
\end{eqnarray*}
\item
By definition,
\begin{eqnarray*}
\aorb_{\alpha+1}(x)=\bigcap\limits_{\substack{U\in\mathcal{U}(x)\text{ open},\\ \sorb_\alpha(U)\subseteq U}}\overline{U}^{\sorb_\alpha}\subseteq \bigcap\limits_{\substack{U\in\mathcal{U}(M)\text{ open},\\ \sorb_\alpha(U)\subseteq U}}\overline{U}^{\sorb_\alpha}=M.
\end{eqnarray*}
\end{enumerate}
\end{proof}

Remark \ref{liapincl} and Lemma \ref{liapthm} yield the following result.
\begin{theorem}
\label{absliap}
The finest decomposition into absolutely Lyapunov stable sets is induced by $\fix\T$.
\end{theorem}

\begin{proof}
We first show that the maximal level sets of $\fix\T$ are absolutely Lyapunov stable. By Remark \ref{liapincl} it suffices to show that a maximal level set $M$ is Lyapunov stable of degree $\alpha$ where $L\cong K/{\sim_\alpha}$ for the fixed factor $L$. Let $x\in K$ such that $M=\pi^{-1}([x])$ where $\pi\colon K\to K/{\sim_\alpha}$ denotes the canonical projection. We first show that for all closed $V\in\mathcal{U}([x])$ there is some open $U\subseteq K$ with $\sorb_\alpha(U)\subseteq U$ such that $\pi^{-1}(V)=\overline{U}^{\sorb_\alpha}$.

Take $U:=\pi^{-1}(V)^\circ$. Clearly, U is open and $\sorb_\alpha(U)=U$ because of $\sorb_\alpha(U)=\pi^{-1}(\pi(U))$. We show that $\pi^{-1}(V)=\overline{U}^{\sorb_\alpha}$. By continuity of $\pi$, we have \[\pi^{-1}(V)\subseteq\overline{\pi^{-1}(V)}^{\sorb_\alpha}=\overline{\pi^{-1}(V)^\circ}^{\sorb_\alpha}=\overline{\pi^{-1}(V^\circ)}^{\sorb_\alpha}=\overline{U}^{\sorb_\alpha}.\] Conversely, \[\overline{U}^{\sorb_\alpha}=\overline{\pi^{-1}(V)}^{\sorb_\alpha}=\bigcap\limits_{\substack{F\text{ closed,}\\\pi^{-1}(V)\subseteq F,\\\sorb_\alpha(F)\subseteq F}}F\overset{F:=\pi^{-1}(V)}{\subseteq} \pi^{-1}(V).\]

By this we obtain

\begin{eqnarray*}
\bigcap\limits_{\substack{U\in\mathcal{U}(M)\text{ open,}\\\sorb_\alpha(U)\subseteq U}}\overline{U}^{\sorb_\alpha}\subseteq\bigcap_{V\in\mathcal{U}([x])\text{ closed}}\pi^{-1}(V)=\pi^{-1}\left(\bigcap_{V\in\mathcal{U}([x])\text{ closed}}V\right)\overset{K/{\sim_\alpha}\text{ Hausdorff}}{=}\pi^{-1}([x])=M. 
\end{eqnarray*}

The converse inclusion \[M\subseteq\bigcap\limits_{\substack{U\in\mathcal{U}(M)\text{ open,}\\\sorb_\alpha(U)\subseteq U}}\overline{U}^{\sorb_\alpha}\] shows that $M$ is Lyapunov stable of degree $\alpha$.

That there is no finer decomposition into absolutely Lyapunov stable sets follows from Lemma \ref{liapthm}.
\end{proof}

As a final result, we link absolute Lyapunov stability and topological ergodicity.

\begin{theorem}
A topological dynamical system $(K;\varphi)$ is topologically ergodic if and only if there is no nontrivial decomposition of $K$ into absolutely Lyapunov stable sets.
\end{theorem}

\end{document}